\documentclass[11pt]{amsart}

\usepackage{amsmath}
\usepackage{amssymb}
\usepackage{graphicx}

\hoffset = - 0.5in

\newtheorem{theorem}{Theorem}[section]
\newtheorem{corollary}[theorem]{Corollary}

\theoremstyle {definition}

\numberwithin{equation}{section}

\renewcommand{\geq}{\geqslant}
\renewcommand{\leq}{\leqslant}

\title{Uniformly nonsquare Banach spaces have the fixed point property 3}

\author[Tim Dalby]{Tim Dalby}

\date{\today}

\keywords{fixed point property, uniformly nonsquare, nearly uniformly smooth}

\subjclass[2010]{46B10, 47H09, 47H10}

\email{tim\_dalby@bigpond.com}

\begin{document}

\parindent = 0pt
\parskip = 8pt

\begin{abstract}

Another proof that uniformly nonsquare Banach spaces have the fixed point property is presented.

\end{abstract}

\maketitle

\section{Introduction}

 In [5] Garc\'{i}a-Falset, Llorens-Fuster and Mazcu\~{n}an-Navarroa were the first to show that a uniformly nonsquare Banach space has the Fixed Point Property, FPP.  The proof is not direct as it uses the following chain of reasoning.  The notation and definitions are explained in the next section.

A Banach space, $X$, is uniformly nonsquare if and only if  $\rho'_X(0) < 1$, see [8].  This then implies that $X$ is reflexive and  $\Gamma_X'(0) < 1$ which in turn implies $RW(a, X) < 1 + a$ for some $a > 0$.  Finally, it was shown that $R(a, X) \leq RW(a, X)$ which allowed a result of Dom\'{i}nguez-Benavides, [3], to be used to obtain the fixed point property.

In summary, the proof goes along the following path. The reverse implications left out for ease of reading.

$X$ is uniformly nonsquare $\Rightarrow \rho'_X(0) < 1 \Rightarrow X \mbox{ is reflexive and } \Gamma_X'(0) < 1 \Rightarrow RW(a, X) < 1 + a \Rightarrow R(a, X) < 1 + a \Rightarrow X \mbox{ has the } FPP.$

This paper explores this idea in a little more detail.  In corollary 4.3 of [5] it was shown that, in a reflexive Banach space $X,$  $\Gamma_X'(0) < 1$ is equivalent to there exist $a > 0$ such that $RW(a, X) < 1 + a.$  This, in turn, is equivalent to $RW(a, X) < 1 + a$ for all $a > 0.$  See [1] for the proof and more details.  So taking $a$ equal a specific value like 1 will suffice.  In other words, in a reflexive Banach spaces, $\Gamma_X'(0) < 1$ is equivalent to $RW(1, X) < 2.$  To try to obtain an insight into the connection between uniformly nonsquare Banach spaces and the FPP, the proof below starts with $RW(1, X) < 2$ and ends with the FPP.  More background can be found in W. A. Kirk and B. Sims, [9].

It is important to make it clear that there is now a very direct proof of uniformly nonsquare implying the FPP courtesy of Dowling, Randrianantoanina and Turett, [4].

\section{Definitions}

Now for the definitions and notation.

The modulus of uniform smoothness is 

\[\rho_X(t) = \sup \left\{ \dfrac{\| x + t y \| + \| x - ty \|}{2} - 1: x, y \in B_X  \right\}\]

 where $t \geq 0$ and the related coefficient is
\[ \rho'_X(0) = \lim_{t \rightarrow 0^+} \frac{\rho_X(t)}{t}. \]

$X$ is defined to be uniformly smooth whenever $\rho'_X(0) = 0.$

In [8] Kato, Maligranda and Takahashi proved, among other interesting results, that a Banach space, $X,$ is uniformly nonsquare if and only if $\rho'_X(0) < 1.$

The modulus of nearly uniform smoothness is

\[ \Gamma_X(t) = \sup \left\{ \inf_{n > 1} \left( \frac{\| x_1 + tx_n \| + \| x_1 - tx_n \|}{2} - 1 \right) \right\}, \]

where $t \geq 0$ and the supremum is taken over all basic sequences $(x_n)$ in $B_X.$  Dom\'{i}nguez-Benavides introduced this modulus in [2].

Also in [2] it was shown that if $X$ is reflexive then

\[ \Gamma_X(t) = \sup \left\{ \inf_{n > 1} \left( \frac{\| x_1 + tx_n \| + \| x_1 - tx_n \|}{2} - 1 \right): (x_n) \mbox{ in } B_X, x_n \rightharpoonup 0 \right\}. \]

A key result from that paper was that a Banach space, $X,$ is nearly uniformly smooth if and only if $X$ is reflexive and 
\[ \Gamma_X'(0) = \lim_{t \rightarrow 0 ^+} \frac{\Gamma_X(t)}{t} = 0. \]

It is clear from the definitions that $\Gamma_X(t) \leq \rho_X(t) \mbox{ for all } t \geq 0.$

So $\Gamma_X'(0) \leq \rho'_X(0).$ Which means that for a reflexive Banach space, uniformly smooth implies nearly uniformly smooth.

Finally, two moduli that were introduced in the quest to find sufficient conditions for a Banach space to have the weak fixed point property, wFPP.

The first one can be found in [3].

\[ R(a, X) = \sup \left \{ \liminf_{n \rightarrow \infty}\| x_n + x \|: (x_n) \mbox{ in } B_X, x_n \rightharpoonup 0, \| x \| \leq a, D[(x_n)] \leq 1 \right \}. \]

where $D[(x_n)] = \limsup_{n \rightarrow \infty} \limsup_{m \rightarrow \infty} \| x_n - x_m \|.$

The second is a refinement of this modulus and first appeared in [5].

 $RW(a, X) =$
 \[ \sup \left \{ \left ( \liminf_{n \rightarrow \infty} \| x_n + x \| \right ) \wedge \left ( \liminf_{n \rightarrow\infty} \| x_n - x \| \right ): (x_n) \mbox{ in } B_X, x_n \rightharpoonup 0, \| x \| \leq a \right \} \]
 
In that paper it was proved that $R(a, X) \leq RW(a, X).$
 
Note that both $RW(a, X)$ and $R(a, X)$ are nondecreasing functions of $a.$

\section{Results}

\begin{theorem} 
Let $X$ be a Banach space where the dual unit ball, $B_{X^*}$, is $w^*$-sequentially compact. If  $RW(1, X) < 2$ then $X$ has the wFPP.
\end{theorem}

\begin{proof}

Let $X$ be a Banach space with $B_{X^*} \, w^*$-sequentially compact and also assume that $RW(1, X) < 2.$

The usual set up is to assume $X$ does not have the wFPP and arrive at a contradiction.  Using results from Goebel [6], Karlovitz [7] and Lin [10] plus an excursion to $l_\infty(X)/c_0(X)$ and then back to $X$ it can be shown that there exists a sequence with the following properties.
 \[ y_n \rightharpoonup y, \lim_{n \rightarrow \infty} \| y_n \| \mbox{ exists, } D[(y_n)] = \lim_{n \rightarrow \infty} \lim_{m \rightarrow \infty} \| y_n - y_m \| \leq \frac{1}{2} \mbox{ and } \| y \| \leq \frac{1}{2}. \]
 See for example [3].
 
To obtain the required contradiction, we need to show that $\lim_{n \rightarrow \infty} \| y_n \|$ is uniformly away from 1.

Using the weak lower semicontinuity of the norm, for all $m > 0,$ 
\[ \liminf_{n \rightarrow \infty} \| (y_m - y_n) + y \| \geq \| y_m \|. \]
So $\liminf_{m \rightarrow \infty} \liminf_{n \rightarrow \infty} \| (y_m - y_n) + y \| \geq \liminf_{m \rightarrow \infty} \| y_m \| = \lim_{m \rightarrow \infty} \| y_m \|.$\\

Now let $y_n^* \in S_{X^*}$ be such that $y_n^*(y_n) = \| y_n \| \mbox{ for all } n.$  

Because $B_{X^*}$ is $w^*$-sequentially compact we may assume, without loss of generality, that $y_n^*\stackrel{*}{\rightharpoonup} y^* \mbox { where } \| y^* \| \leq 1.$

Again, for all $m > 0,$
 \begin{align*}
\liminf_{n \rightarrow \infty} \| (y_m - y_n) - y \| & \geq \liminf_{n \rightarrow \infty} (-y_n^*)((y_m - y_n) - y))\\
& = \liminf_{n \rightarrow \infty} y_n^*(y_n) - y^*(y_m - y)\\
& = \liminf_{n \rightarrow \infty}\| y_n \| - y^*(y_m - y).
\end{align*}
Therefore, using $y_m - y \rightharpoonup 0,$
\[\liminf_{m \rightarrow \infty}\liminf_{n \rightarrow \infty}\| (y_m - y_n) - y \| \geq \liminf_{m \rightarrow \infty}\| y_m \|=\lim_{m \rightarrow \infty}\| y_m \|.\]
Thus
\[\liminf_{m \rightarrow \infty}\liminf_{n \rightarrow \infty}\| (y_m-y_n) - y \|\wedge \liminf_{m \rightarrow \infty}\liminf_{n \rightarrow \infty}\| (y_m - y_n) + y \| \geq \lim_{m \rightarrow \infty}\| y_m \|.\]

Using lemma 3.2 of [5], there exists a subsequence $(y_{n_k})$ such that
\[ \limsup_{k \rightarrow \infty}\|  y_{n_k} - y_{n_{k + 1}} \| \leq D[(y_n)] \leq \frac{1}{2} \qquad\qquad \dag\]
and 
\[ \liminf_{k \rightarrow \infty} \| y_{n_k} - y_{n_{k + 1}} + y \| \geq \liminf_{m \rightarrow \infty} \liminf_{n \rightarrow \infty} \| (y_n - y_m) + y \| \]
and 
\[ \liminf_{k \rightarrow \infty} \| y_{n_k} - y_{n_{k + 1}} - y \| \geq \liminf_{m \rightarrow \infty} \liminf_{n \rightarrow \infty} \| (y_n - y_m) - y \|. \]
Therefore
\[\liminf_{k \rightarrow \infty}\| y_{n_k} - y_{n_{k+1}} + y \| \wedge \liminf_{k\rightarrow
\infty}\| y_{n_k} - y_{n_{k+1}} - y \| \geq \lim_{m \rightarrow \infty}\| y_m \|.\]

Because $RW(1, X) < 2$ we have $\dfrac{1}{RW(1, X)} > \dfrac{1}{2} \mbox{ and } \dfrac{1}{RW(1, X)} - \dfrac{1}{2} > 0.$  So there exist an $\epsilon$ such that $0 < \epsilon < \dfrac{1}{RW(1, X)} - \dfrac{1}{2}.$

That is, there exists $\epsilon > 0 \mbox{ such that } \epsilon RW(1, X) < 1 - \dfrac{RW(1, X)}{2}$  or \newline $(1/2 + \epsilon)RW(1, X) < 1.$\\

From \dag \, there exists $k_0$ such that
\[ \| y_{n_k}  - y_{n_{k+1}} \| \leq \frac{1}{2} + \epsilon \mbox{ for all } k \geq k_0.\]

Consider $w_k = \dfrac{y_{n_{k_0 + k}} - y_{n_{k_0 + k + 1}}}{1/2 + \epsilon} \mbox{ for } k \geq 0.$
Then $w_k \rightharpoonup 0 \mbox{ and } \| w_k \| \leq 1 \mbox{ for all } k.$

Let $w = \dfrac{y}{ 1/2 + \epsilon}$ then $\| w \| = \dfrac{\| y \|}{1/2 + \epsilon} \leq \dfrac{1/2}{1/2 + \epsilon}.$

$RW\left( \dfrac{1/2}{1/2 + \epsilon},X \right)$
\begin{align*}
 & \geq \liminf_{k \rightarrow \infty}\| w_k + w \| \wedge \liminf_{k \rightarrow \infty}\| w_k - w \| \\
&=\dfrac{1}{1/2 + \epsilon}\left(\liminf_{k \rightarrow \infty}\|y_{n_{k_0 + k}} - y _{n_{k_0 + k + 1}} + y\| \wedge \liminf_{k \rightarrow \infty}\| y_{n_{k_0 + k}} - y_{n_{k_0 + k + 1}} - y\|\right) \\
&\geq \dfrac{1}{1/2 + \epsilon}\left(\liminf_{k \rightarrow \infty}\|y_{n_k} - y_{n_{k + 1}} + y\| \wedge \liminf_{k\rightarrow
\infty}\| y_{n_k} - y_{n_{k + 1}} - y \|\right) \\
&\geq \dfrac{1}{1/2 + \epsilon}\lim_{m\rightarrow \infty}\| y_m \| .
\end{align*}

So, using the property that $RW(a, X)$ is a nondecreasing function of $a,$
\begin{align*}
\lim_{m\rightarrow \infty}\| y_m \| &\leq (1/2 + \epsilon) RW \left(\frac{1/2}{1/2 + \epsilon}, X \right) \\
&\leq (1/2 + \epsilon) RW(1, X )\notag\\
& <1 .
\end{align*}
A contradiction.

\end{proof}

\begin{corollary}
If $X$ is a uniformly nonsquare Banach space then $X$ has the FPP.
\end{corollary}

\begin{proof}
Assume $X$ is a uniformly nonsquare Banach space.  This is equivalent to $\rho'_X(0) < 1.$   Uniformly nonsquare spaces are reflexive and so $B_{X^*}$ is $w^*$-sequentially compact and also $\rho'_X(0) < 1$ implies $\Gamma_X'(0) < 1.$  This last condition is equivalent to $RW(a, X) < 1 + a$ for some $a > 0.$  Finally, this condition is equivalent to $RW(1, X) < 2$ and the previous theorem can be invoked. 
\end{proof}

\end{document}